\documentclass[11pt]{nyjm}

\usepackage{amsmath}
\usepackage{amsthm}
\usepackage{amssymb,MnSymbol,color,tikz}
\usepackage{graphicx}
\usepackage{amssymb}
\usepackage{epstopdf}

\title[Subdivision rule constructions on quadratic matings]{Subdivision rule constructions on critically preperiodic quadratic matings}

\author{Mary Wilkerson}
\address{} 
\email{mwilkerso@coastal.edu}   



\newtheorem{theorem}{Theorem}[section]

\theoremstyle{definition}
\newtheorem{definition}[theorem]{Definition}

\newtheorem{example}[theorem]{Example}

\begin{document} 
 

\begin{abstract}
Mating is an operation that identifies the domains of a polynomial pair in order to obtain a new map on the resulting quotient space. The dynamics of the mating are then dependent on the two polynomials and the manner in which the quotient space was defined, which can be difficult to visualize. This research addresses using Hubbard trees and finite subdivision rules as tools to examine quadratic matings with preperiodic critical points. In many cases, discrete parameter information on such quadratic pairs can be translated into topological information on the dynamics of their mating. The central theorems in this work provide methods for explicitly constructing subdivision rules that model non-hyperbolic matings. We follow with several examples and connections to the current literature.
\end{abstract}

\maketitle
\tableofcontents


\section{Introduction}

Even the simplest of rational maps can have surprisingly complicated dynamics. Many rational maps may exhibit polynomial-like behavior though, which is better understood. In the early 1980's, Douady described \emph{polynomial mating}---a way to combine two polynomials in order to obtain a new map with shared dynamics from both original maps \cite{DOUADY}. Frequently, this mating is dynamically similar to a rational map. When rational maps are topologically conjugate to matings, we can examine the dynamics of the constituent polynomials in the mating to better understand the rational map.

So, what is a mating? Suppose we consider the compactification $\widetilde{\mathbb{C}}$ of $\mathbb{C}$ given by adding in the circle at infinity, $\widetilde{\mathbb{C}}=\mathbb{C}\cup\{\infty\cdot e^{2\pi i \theta}|\theta \in \mathbb{R}/\mathbb{Z}\}$. Then, we take two polynomials of the same degree with connected filled Julia sets acting on two disjoint copies of $\widetilde{\mathbb{C}}$. If we use these domains to form a quotient space in an appropriate manner, our polynomial pair will determine a map that descends to this new space. The map on the quotient space is called a \emph{mating} of the two polynomials. The different kinds of polynomial matings are each dependent upon how we identify points on our copies of $\widetilde{\mathbb{C}}$. (While we provide more details later, an excellent overview of some fundamental mating constructions is given in \cite{MATINGS}.)

In a topological mating, the domain is given by a quotient space which identifies the boundaries of two filled Julia sets. The resulting domain can sometimes be surprising: by results of Lei, Rees, and Shishikura, it is possible to develop an equivalence relation on two connected filled Julia sets---including ones with no interior---such that the associated quotient space is a topological two-sphere \cite{MATINGSOFQUADRATICS},\cite{REES},\cite{SHISHIKURA}.  In \cite{MEDUSA}, a general method is presented for developing the mating resulting from a given polynomial pairing---but this method is best suited for the hyperbolic case. As visualization of how the boundary identifications develop can be useful, this paper presents an option for the case of two critically preperiodic polynomials.  

The combinatorial construction given in this paper develops polynomial parameters into more extensive information on a mating by using a discrete model. We do this by looking at a simplified, combinatorial model of the Julia set---i.e., the \emph{Hubbard tree}--- for both of the polynomials we intend to mate. We examine the identifications between Hubbard trees that occur in forming the quotient space for the essential mating, and how these can be used to obtain a 1-skeleton for the tiling of a finite subdivision rule. The author expands here upon preliminary results given in \cite{DISSERTATION}.

In \S \ref{prereqs}, we detail the prerequisites needed to define and construct polynomial matings. We also describe finite subdivision rules and Hubbard trees, and why their use is relevant here. 

In \S \ref{construction} we introduce the essential construction for obtaining finite subdivision rules from matings, and demonstrate using several examples. We then close with connections to the current literature and future avenues for exploration in \S \ref{connections}. 

\section{Prerequisites}\label{prereqs}

\subsection{Parameter Space}\label{parameter} Suppose that $c$ is contained in the Mandelbrot set, and that $K_c$ is the filled Julia set of the map $f_c(z)=z^2+c$. Since this implies that $K_c$ is connected, $\widehat{\mathbb{C}}\backslash K_c$ is conformally isomorphic to the complement of the closed unit disk via some holomorphic map $\phi: \widehat{\mathbb{C}}\backslash\overline{\mathbb{D}} \rightarrow \widehat{\mathbb{C}}\backslash K_c$. The map $\phi$ can be chosen to 
conjugate $z\mapsto z^2$ on $\widehat{\mathbb{C}}\backslash\overline{\mathbb{D}}$ to $f_c$ on $\hat{\mathbb{C}}\backslash K_c$ so that $\phi(z^2) = f_c(\phi(z))$, in which case $\phi$ is a unique map.

Taking the image of rays of the form $\{re^{2\pi it} :r\in(1,\infty)\}$ under $\phi$ for fixed $t\in \mathbb{R}/\mathbb{Z}$ then yields the \emph{external ray} of angle $t$, $R_c(t)$. (See Figure \ref{varphi}.) If $K_c$ is locally connected, the map $\phi$ extends continuously to a map from the unit circle to the Julia set $J_c$ and external rays of angle $t$ are said to \emph{land} at the point $\gamma(t) = \displaystyle\lim_{r\rightarrow1^+} \phi(re^{2\pi it})$. The map $\gamma:\mathbb{R}/\mathbb{Z}\rightarrow{J_c}$ is called the \emph{Carath\'{e}odory semiconjugacy}, with the associated identity $$\gamma(2\cdot t)=f_c(\gamma(t))$$ in the degree 2 case. This identity allows us to easily track forward iteration of external rays and their landing points in $J_c$ by doubling the angle of their associated external rays modulo 1.

\begin{figure}[hbt]
\center{\includegraphics{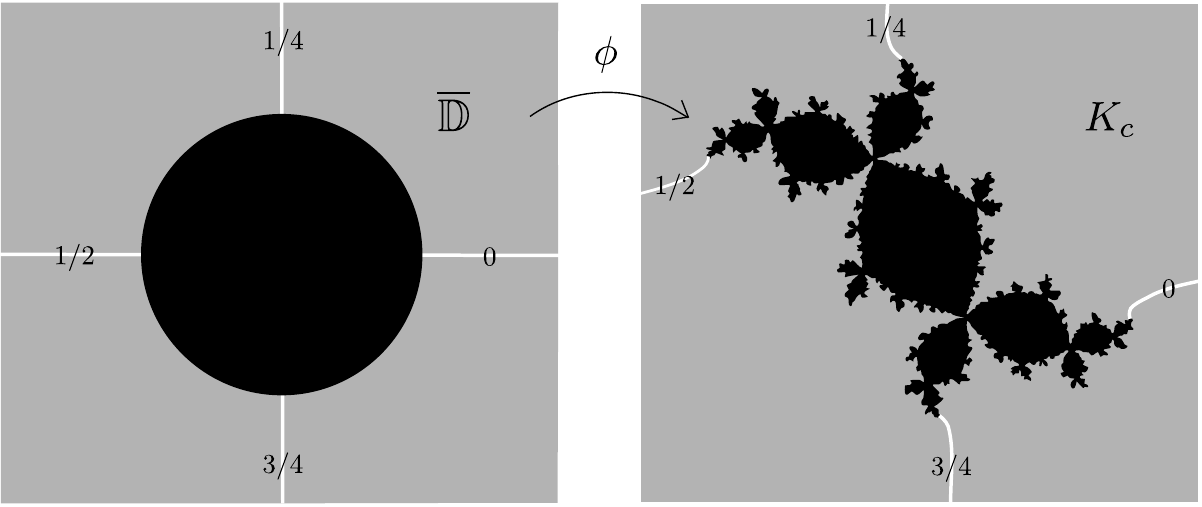}}
\caption{The conformal isomorphism $\phi$ and selected external rays on the rabbit polynomial.}
\label{varphi}
\end{figure}

The work in this paper will be restricted to the use of polynomials whose parameters are obtained from  \emph{Thurston-Misiurewicz} points--values of $c$ on the boundary of the Mandelbrot set at which the critical point of $f_c$ is strictly preperiodic. Critically preperiodic polynomials are typically parameterized by the angle $\theta$ of the external ray landing at the critical value rather than by the critical value. We will follow this convention from this point on, using $f_\theta$ in lieu of $f_c$. These critically preperiodic polynomials have filled Julia sets that are \emph{dendrites}: locally connected continuums that contain no simple closed curves. In other words, the filled Julia set of such a polynomial possesses a possibly infinite tree-like structure, has no interior, and is the Julia set of the polynomial. Further, since these $f_\theta$ have Julia sets that are locally connected, recall that external rays land on $J_\theta$. This means that the conformal isomorphism $\phi$ and Carath\'{e}odory semiconjugacy $\gamma$ can be used to recover the mapping behavior of $f_\theta$ on its Julia set. As a brief example, consider Figure \ref{labelrays}: we could obtain that the critical orbit is preperiodic and follows the pattern $c_0\mapsto c_1 \mapsto c_2\mapsto c_3\mapsto c_2$ by evaluation in $f_{1/6}$, or we could double the angles of external rays landing at these points to obtain the same pattern.

\begin{center}\end{center}

\begin{figure}[htb]
\center{\includegraphics{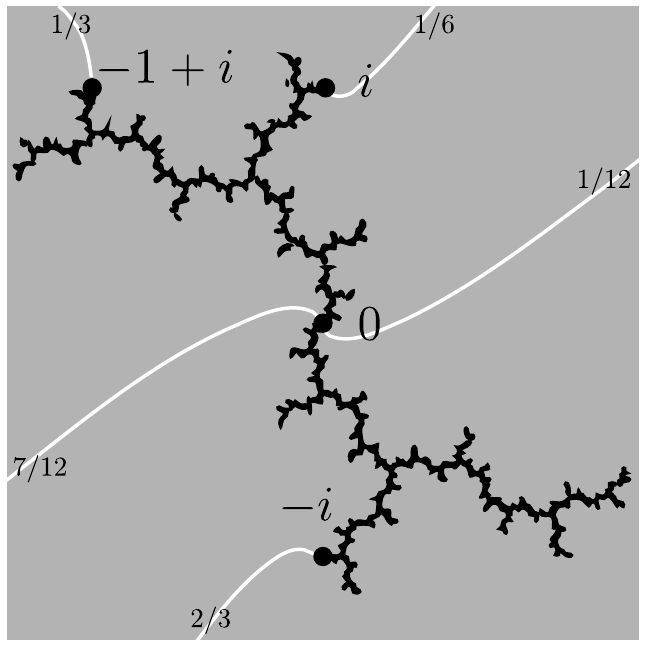}}
\caption{External rays landing on the critical orbit of $f_{1/6}(z)=z^2+i$.}
\label{labelrays}
\end{figure}

\subsection{Matings}\label{mating}  

Let $f_\alpha:\widetilde{\mathbb{C}}_\alpha\rightarrow\widetilde{\mathbb{C}}_\alpha$ and $f_\beta:\widetilde{\mathbb{C}}_\beta\rightarrow\widetilde{\mathbb{C}}_\beta$ be postcritically finite monic quadratic polynomials taken on two disjoint copies of $\widetilde{\mathbb{C}}$. Form the topological two-sphere $\mathbb{S}^2$ by taking $\mathbb{S}^2=\widetilde{\mathbb{C}}_\alpha\bigsqcup \widetilde{\mathbb{C}}_\beta/\sim_f$, where $\sim_f$ identifies $\infty \cdot e^{2\pi i t}$ on $\widetilde{\mathbb{C}}_\alpha$ with $\infty \cdot e^{-2\pi i t}$ on $\widetilde{\mathbb{C}}_\beta$. This yields a topological two-sphere by gluing two copies of $\tilde{\mathbb{C}}$ together along their circles at infinity with opposing angle identifications. (See Figure \ref{green}.) This quotient space serves as the domain of the \emph{formal mating} $f_\alpha\upmodels_ff_\beta$, which is the map that applies $f_\alpha$ and $f_\beta$ on their respective hemispheres of $\mathbb{S}^2$. The Carath\'{e}odory semiconjugacy guarantees that $f_\alpha\upmodels_ff_\beta$ is well-defined on the equator and provides a continuous branched covering of $\mathbb{S}^2$ to itself. We will use $F=f_\alpha\upmodels_ff_\beta$ to denote the formal mating whenever it is unambiguous to do so.

\begin{figure}[h]
\center{\includegraphics{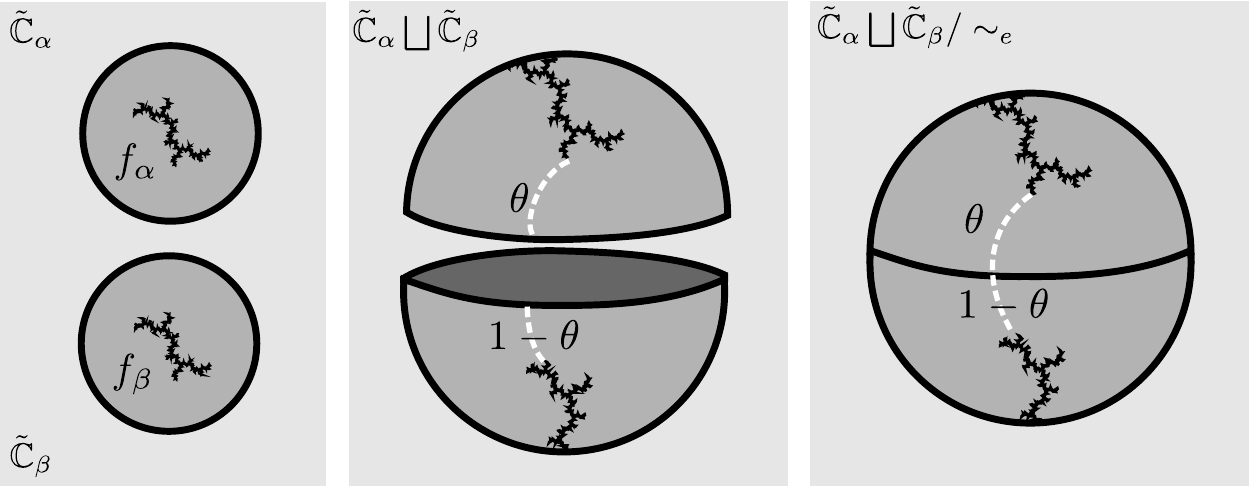}}
\caption{Steps in the formation of the formal mating.}
\label{green}
\end{figure}

The \emph{topological mating} $f_\alpha\upmodels_tf_\beta$, on the other hand, is formed by using the quotient space $K_\alpha\bigsqcup K_\beta/\sim_t$, where $\sim_t$ identifies the landing point of $R_\alpha(t)$ on $J_\alpha$ with the landing point of $R_\beta(-t)$ on $J_\beta$. This glues the Julia sets of $f_\alpha$ and $f_\beta$ together at opposing external angles. Similar to the formal mating, we obtain the map $f_\alpha\upmodels_tf_\beta$ by applying $f_\alpha$ and $f_\beta$ on their respective filled Julia sets. The Carath\'{e}odory  semiconjugacy similarly guarantees that the resulting map is well-defined and continuous, but it is possible that it no longer acts on a quotient space which is a topological two-sphere.

The quotient space obtained in developing the topological mating sometimes is a two-sphere, however---and further, $f_\alpha\upmodels_tf_\beta$ may be topologically conjugate to a rational map on the Riemann sphere. Such a rational map is called the \emph{geometric mating} of $f_\alpha$ and $f_\beta$. The following elegant result highlights a case that we will restrict our examination to in this paper: 

\begin{theorem}[Lei, Rees, Shishikura] The topological mating of the postcritically finite maps $z\mapsto z^2+c$ and $z\mapsto z^2+c'$ is Thurston equivalent to a rational map on $\hat{\mathbb{C}}$ if and only if $c$ and $c'$ do not lie in complex conjugate limbs of the Mandelbrot set \cite{MATINGSOFQUADRATICS}, \cite{REES}, \cite{SHISHIKURA}. 

\end{theorem}

This is useful since it allows us to determine if a polynomial mating will have a domain given by a quotient two-sphere based on parameters alone. Given that this two-sphere is obtained by identifying the boundaries of two Julia sets, and that one or both of these Julia sets may be dendrites though, this result may appear somewhat counterintuitive.

To assist in understanding how boundary identifications come together in the quotient space of the topological mating, we will examine the \emph{essential} mating, $f_\alpha\upmodels_ef_\beta$. (Similar to our convention for the formal mating, we will use $E=f_\alpha\upmodels_ef_\beta$ to denote the essential mating whenever it is unambiguous to do so.) Starting with the quotient two-sphere $\mathbb{S}^2$ developed in the formal mating $F$, the essential mating is constructed as detailed below and in \cite{MATINGSOFQUADRATICS}.

\begin{definition}\label{thm:essential} Let $\{l_1,...,l_n\}$ be the set of connected graphs of external rays on $\mathbb{S}^2$ containing at least two points of the postcritical set$P_F$, and let $\{\tau_1,...,\tau_m\}$ be the set of connected graphs of external rays in $\displaystyle\bigcup_{k\in\mathbb{N}}\bigcup_{i=1}^nF^{-k}(l_i)$ containing at least one point on the critical orbit of $F$. Take each of the $\{\tau_1,...,\tau_m\}$ to be an equivalence class of the equivalence relation $\sim_e$. Note that $\mathbb{S}'^2=\mathbb{S}^2/\sim_e$ is homeomorphic to a sphere. Further, $F$ maps equivalence classes to equivalence classes, so letting $\pi:\mathbb{S}^2\rightarrow\mathbb{S}'^2$ denote the natural projection yields that $\pi\circ F\circ\pi^{-1}$ is well-defined and preserves the mapping order of the equivalence classes $\{\tau_1,...,\tau_m\}$.

This composition is not a branched covering, though. To rectify this, set $V_j$ to be an open neighborhood of $\tau_j$ such that $V_j\cap(P_F\cup\Omega_F)=\tau_j\cap(P_f\cup\Omega_F)$ for each $j$, and such that distinct $V_j$ are nonintersecting. For each $j$, denote by $\{U_{ij}\}$ the set of connected components of $F^{-1}(V_j)$ for which $U_{ij}\cap\displaystyle\bigcup_{p=1}^m\tau_p=\emptyset$. 

Finally, we set $E:\mathbb{S}'^2\rightarrow\mathbb{S}'^2$ as equivalent to $\pi\circ F\circ\pi^{-1}$ off of the set $\displaystyle\bigcup_{i,j}\pi(U_{ij})$, and for each $i,j$ set $E:\pi(U_{ij})\rightarrow\pi(V_j)$ to be a homeomorphism that extends continuously to the boundary of each $\pi(U_{ij})$.  $E$ is the \emph{essential mating} of $f_\alpha$ and $f_\beta$.
\end{definition} 

Despite the appearance of $E$ being defined rather arbitrarily in the last step, the essential mating is uniquely determined up to Thurston equivalence, and is in fact a degree 2 branched covering map which is Thurston-equivalent to the associated topological mating. In a sense, the essential mating captures the ``essential" identifications--i.e., mostly ones on the critical orbit--that are made in forming the topological mating. $E$ mostly behaves like the map $F$, with the fundamental difference being that the domain and range of $E$ are a quotient space where these important identifications on the critical orbit of $F$ are collapsed together. It should thus be noted that if no postcritical points of $F$ can be connected by a graph of adjacent external rays on $\mathbb{S}^2$, then $\sim_e$ is the trivial equivalence relation and the essential mating is the formal mating.

 \subsection{Finite Subdivision Rules}\label{fsr} Our ultimate motivation in examining the essential mating is to develop a tiling construction that highlights the identifications formed in the topological mating. We will develop this construction using \emph{finite subdivision rules}. 
 
 \begin{definition} A \emph{finite subdivision rule} $\mathcal{R}$ consists of the following three components: 

\begin{enumerate}

\item A tiling. Formally, this is a finite 2-dimensional CW complex $S_\mathcal{R}$, called the subdivision complex, with a fixed cell structure such that $S_\mathcal{R}$ is the union of its closed 2-cells. We assume that for each closed 2-cell $\tilde{s}$ of $S_\mathcal{R}$ there is a CW structure $s$ on a closed 2-disk such that $s$ has $\geq 3$ vertices, the vertices and edges of $s$ are contained in $\partial s$, and the characteristic map $\psi_s:s\rightarrow S_\mathcal{R}$ which maps onto $\tilde{s}$ restricts to a homeomorphism on each open cell.

\item A subdivided tiling. Formally, this is a finite 2-dimensional CW complex $\mathcal{R}(S_\mathcal{R})$ which is a subdivision of the above CW complex $S_\mathcal{R}$.

\item A continuous cellular map $g_\mathcal{R}: \mathcal{R}(S_\mathcal{R})\rightarrow S_\mathcal{R}$, called the subdivision map, whose restriction to any open cell is a homeomorphism. \textup{\cite{FSRS} }

\end{enumerate}
\end{definition}

In essence, a finite subdivision rule is a finite combinatorial rule for subdividing tilings on some 2-complex. We restrict, however, to tilings formed by ``filling in" connected finite planar graphs on a two-sphere with open tiles that are topological polygons. None of these tiles are allowed to be monogons or digons, and further, each edge of the tiling must be a boundary edge to some tile. These tiles may be non-convex, though---to the potential extreme of allowing both sides of a single edge to form two sides of the boundary of a single tile. (For example, a line segment with both end points and the midpoint marked on the two-sphere forms the boundary of a topological quadrilateral.) 

Once we subdivide a tiling, we will need a map that takes open cells of the subdivision tiling homeomorphically to open cells of the original tiling. Only when we have all three components---the initial tiling, the subdivision tiling, and a subdivision map---do we have a complete finite subdivision rule. Then, this rule can be applied recursively to yield iterated subdivisions of the original tiling.

\begin{example} Consider Figure \ref{fig:fsr}: $\hat{\mathbb{C}}$ is oriented so that the marked points $0, \pm1,$ and $\infty$ all lie on the equator. The equator and marked points determine a graph which yields a tiling of $\hat{\mathbb{C}}$ into two topological quadrilaterals. If we take a preimage of this structure under the map $z\mapsto z^2$, we obtain a tiling that has four quadrilaterals---each of which maps homeomorphically onto one of the quadrilaterals in the original tiling. Here, the structure on the left is our tiling, the structure on the right is the subdivided tiling, and the map $z\mapsto z^2$ is the subdivision map.

\end{example}

While a finite subdivision rule may be defined using analytic maps and embedded tilings as in the previous example, this is not necessary. We can use the mapping behavior of $n$-cells in a tiling to determine the mapping behavior of ($n+1$)-cells, thus obtaining a subdivision map based on combinatorial data. The reader may reference Cannon, Floyd, and Parry in \cite{FSRS} for a more detailed treatment of this topic.

\begin{figure}[htb]
\center{\includegraphics{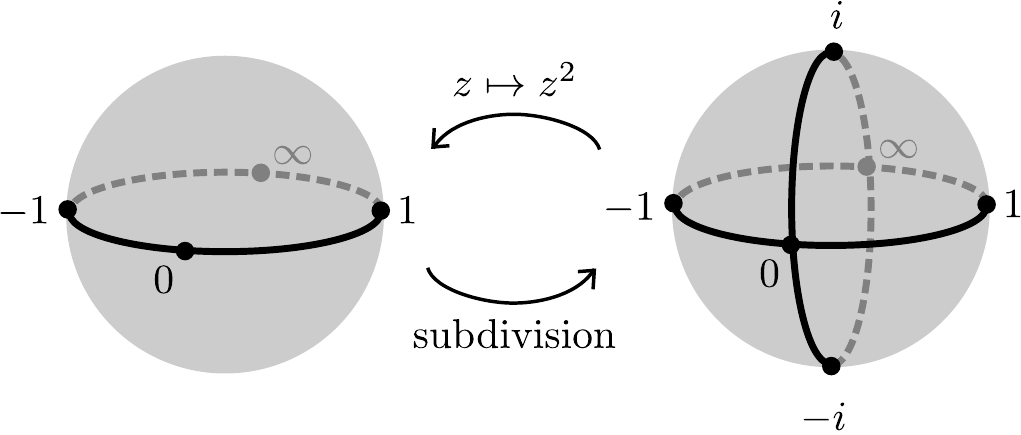}}
\caption{A rudimentary tile subdivision.}
\label{fig:fsr}
\end{figure}

 \subsection{Hubbard trees}\label{Hubbard}

In order to build a finite subdivision rule later on, it will be helpful to have a finite invariant structure in mind to determine the tiling. The Julia set is invariant under iteration of its associated polynomial, but the structure of the Julia set is more complicated than we would like to use as a starting point for a finite subdivision rule. Thus, we would like to work with a discrete approximation to the Julia set: the \emph{Hubbard tree}.

(Note: Hubbard Trees are defined in \cite{SQUARE1} using \emph{allowable arcs}. The construction of an allowable arc is simplified considerably for the case where $f$ has a dendritic Julia set, so for the reader's convenience we present a definition restricted to this case here.)

\begin{definition} Let $f_\theta: \mathbb{C} \rightarrow \mathbb{C}$ be given by $f_\theta(z) = z^2 + c$ for some Misiurewicz point $c$, and let $f_\theta$ have Julia set $J_\theta$ and postcritical set $P_{f_\theta}$.

We say that a subset $X$ of $J_\theta$ is \emph{allowably connected} if $x,y\in X$ implies that there is a topological arc in $X$ that connects $x$ and $y$. The \emph{allowable hull} of a subset $A$ in $J_\theta$ is then the intersection of all allowably connected subsets of $J_\theta$ which contain $A$. Finally, the \emph{Hubbard tree} of $f_\theta$ is the allowable hull of $P_{f_\theta}$ in $J_\theta$.
\end{definition}

\begin{figure}[htb]
\center{\includegraphics{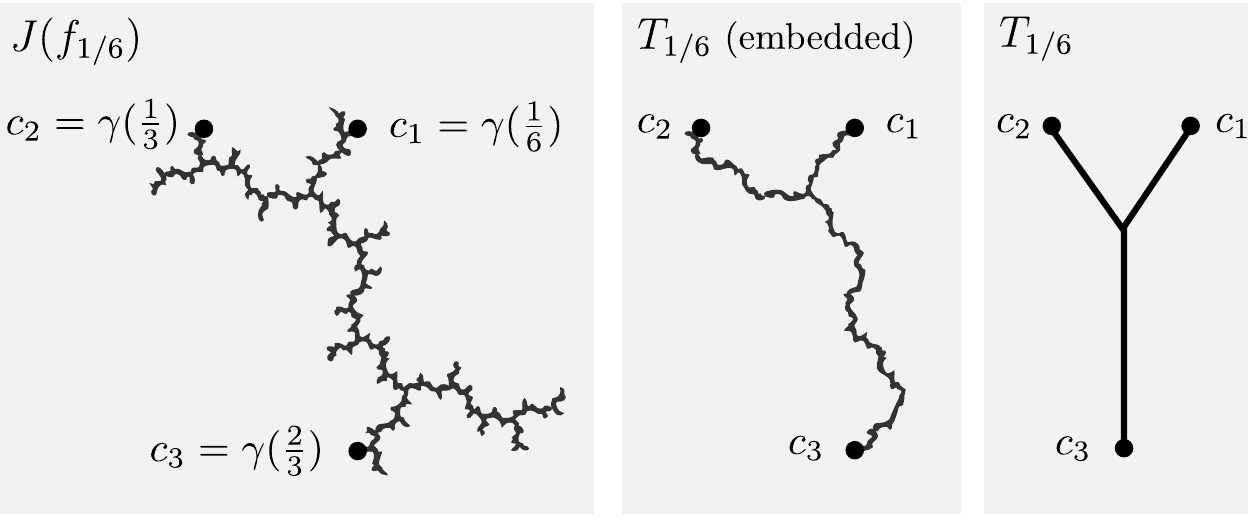}}
\caption{The Julia set and Hubbard trees for $f_{1/6}(z)=z^2+i$.}
\label{hubbardtree}
\end{figure}

The Hubbard tree as defined above is embedded in $\mathbb{C}$ and topologically equivalent to the notion of an \emph{admissible Hubbard tree} with preperiodic critical point as discussed in \cite{HUBBARDTREES}. The notes of Bruin and Schleicher's in \cite{HUBBARDTREES}, however, emphasize the combinatorial structure of the Hubbard tree as a graph with vertices marked by elements of $P_\theta$, rather than as an embedded object in the complex plane. (See Figure \ref{hubbardtree}.) They present several explicit algorithms that can be used to construct a topological copy of $T_\theta$ from the parameter $\theta$, building heavily on the notion that quadratic maps are local homeomorphisms off of their critical points, and degree two at their critical points. We can further expand upon these observations regarding the behavior of quadratic polynomials to determine how forward images and preimages of the Hubbard tree $T_\theta$ under $f_\theta$ will present: forward images are invariant and map the tree onto itself, every point in $T_\theta$ has at most two inverse images under $f_\theta$, $f_\theta$ acts locally homeomorphically on $T$ everywhere except at the critical point, and subsequent preimages  of $T_\theta$ under $f_\theta$ give discrete approximations to $J_\theta$. (The $n$th preimage of an tree $T$ under its associated polynomial $f$ contains $2^n$ miniature copies of the tree which each map homeomorphically onto the tree via $f^{\circ n}$, as in Figure \ref{hubbardpreim}.) In addition, Hubbard trees have many desirable characteristics that we will later require the 1-skeletons of subdivision complexes to possess--namely, being planar, finite,  forward invariant, and containing the postcritical set.

\begin{figure}[htb]
\center{\includegraphics{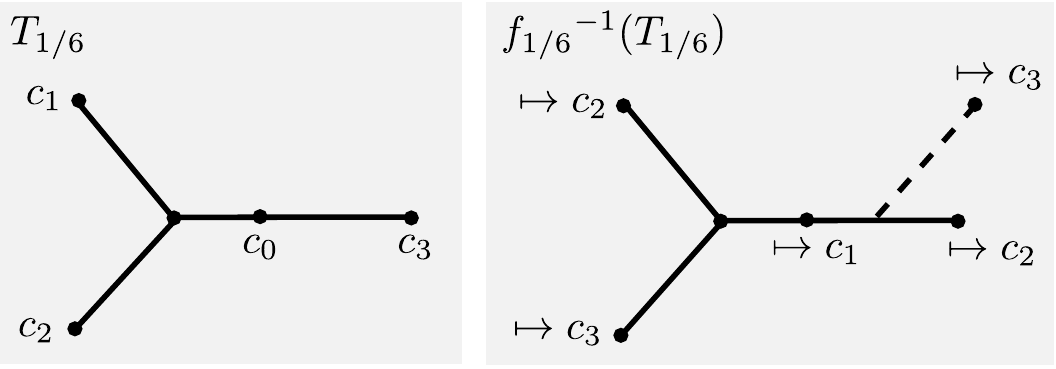}}
\caption{\label{hubbardpreim}{Preimages of a Hubbard tree under its associated polynomial.}}
\end{figure}

\section{An essential finite subdivision rule construction}\label{construction}

Recall that the emphasis for this paper is on the non-hyperbolic case in which two postcritically finite polynomials with dendritic Julia sets are mated. If we further restrict our work to the setting where the critical values of these polynomials are not in complex conjugate bulbs of the Mandelbrot set, the topological mating is Thurston-equivalent to a rational map on the Riemann sphere.  In order to understand how the quotient space for the mating comes together, we will construct a combinatorial model of the mating  in the form of a finite subdivision rule.

\subsection{The essential construction.}

An ideal finite subdivision rule should be based upon a subdivision map that is dynamically similar to the topological and geometric matings.  The formal mating will not always suffice: if any postcritical points of $F$ are contained in the same equivalence class of $\sim_t$, $F$ is not Thurston-equivalent to the topological mating. On the other hand, the essential mating \emph{does} give Thurston-equivalence to the topological and geometric matings---thus, it is a desirable subdivision map.

This leaves us to determine the tiling and subdivided tiling for a given mating. The Hubbard trees associated with the polynomial pair for our mating are a good start for a tiling 1-skeleton, as they record much of the dynamic information associated with the polynomials. However, there are two trees associated with any mating, and we need to reconcile this structure on $\mathbb{S}^2/\sim_e$. For many polynomial pairs, this problem solves itself quite naturally:

\begin{definition}[Finite subdivision rule construction, essential type] Let $f_\alpha$ and $f_\beta$ be critically preperiodic monic quadratic polynomials such that $x\sim_e y$ for some points $x\in T_\alpha$, $y\in T_\beta$.

Give $T_\alpha \bigsqcup T_\beta/\sim_e$ a graph structure on the quotient space of the essential mating by marking all postcritical points and branched points as vertices. (If need be, mark additional periodic or preperiodic points on $T_\alpha$ or $T_\beta$ and the points on their forward orbits to avoid tiles forming digons.) The associated 2-dimensional $CW$ complex for this structure will yield the subdivision complex, $S_\mathcal{R}$.

Select a construction of the essential mating $E$ and set $\mathcal{R}(S_\mathcal{R})$ to be the preimage of $S_\mathcal{R}$ under $E$, taking preimages of marked points of  $S_\mathcal{R}$ to be marked points of $\mathcal{R}(S_\mathcal{R})$ .

If $\mathcal{R}(S_\mathcal{R})$ is a subdivision of $S_\mathcal{R}$ and if the essential mating $E:\mathcal{R}(S_\mathcal{R})\rightarrow S_\mathcal{R}$ is a subdivision map, then $\mathcal{R}$ is a finite subdivision rule and the above construction is labelled of \emph{essential type}.
\end{definition}

The central idea behind this approach is that groupings of points on the critical orbit of $F$ which are identified under $\sim_e$ must be collapsed if we wish to use the essential mating as a subdivision map. The quotient of $T_\alpha\bigsqcup T_\beta$ under $\sim_e$ is a connected graph when $\sim_e$ is associated with a nontrivial essential mating, as in the example in Figure \ref{fig:f16f16i}. If we ``fill in" the open spaces of this graph with polygonal tiles, we obtain a subdivision complex $S_\mathcal{R}$ which in many cases subdivides when we consider its pullback by $E$.  We formalize these notions with the following theorem:

\begin{figure}[htb]
\center{\includegraphics{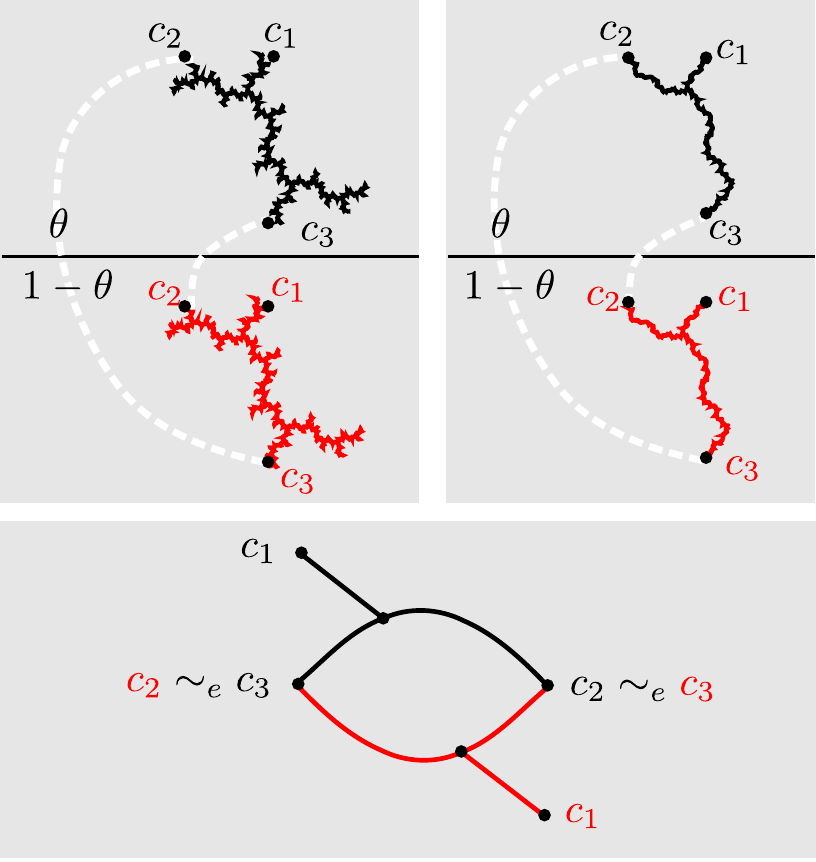}}
\caption[The rays shown here collapse under $\sim_e$]{\label{fig:f16f16i}{External ray-pairs which connect the periodic postcritical points of $f_{1/6}\upmodels_ff_{1/6}$ also modeled on Hubbard trees. The rays shown here collapse under $\sim_e$.}}
\end{figure}

\begin{theorem}\label{thm:construc1}
Let $F$ be the formal mating of $f_\alpha$ and $f_\beta$. The essential type construction fails to yield a finite subdivision rule generated by this polynomial pairing if and only if there exists some $x,y$ in $T_\alpha\bigsqcup T_\beta$ with $x\sim_t y$, $x\not\sim_e y$, and $F(x)\sim_e F(y)$.
\end{theorem}

\begin{proof}

We prove the backward direction by contradiction. Using the notation developed in Definition \ref{thm:essential} for the essential mating, if such an $x$ and $y$ exist, we must have that $x,y \in U_{ij}$ with $F(x),F(y)\in V_j$ for  some $i,j$. Recall that the essential mating $E$ is then a homeomorphism from $\pi(U_{ij})$ to $\pi(V_j)$. 

Since $F(x)\sim_eF(y)$, we can choose $V_j$ so that it contains no other marked points of our 1-skeleton, and so that $\pi(V_j)$ intersected with the 1-skeleton of $S_\mathcal{R}$ yields a connected subset of $\mathbb{S}^2$. $E$ being a homeomorphism then implies that the 1-skeleton of $E^{-1}(S_\mathcal{R})=\mathcal{R}(S_\mathcal{R})$ intersected with $\pi(U_{ij})$ is connected. $\pi(U_{ij})$ intersected with the 1-skeleton of $S_\mathcal{R}$, however will not be connected since $x\not\sim_ey$. This suggests that at least one edge must have been added to the 1-skeleton of $\mathcal{R}(S_\mathcal{R})$ in this neighborhood during a subdivision of $S_\mathcal{R}$. Thus, the intersection of $U_{ij}$ with $\mathcal{R}(S_\mathcal{R})$ should have at least two marked points corresponding to the endpoints of this edge (and potentially others) added during the subdivision of $S_\mathcal{R}$. This cannot be so, however, since by the construction this intersection should contain only the single marked point $E^{-1}\circ\pi\circ F(x)$. Thus, the construction does not yield a finite subdivision rule in this case.

We now prove the forward direction by contrapositive: suppose that there exist no $x,y$ in $T_\alpha\bigsqcup T_\beta$ with $x\sim_ty$, $x\not\sim_e y$, and $F(x)\sim_eF(y)$. Then for every $U_{ij}$, at least one of $U_{ij}\cap T_\alpha$ or $U_{ij}\cap T_\beta$ must be $\varnothing$. We will now use $E$ to denote the essential mating formed with the additional restrictions that $E|_{U_{ij}\cap\pi(T_\alpha\bigsqcup T_\beta)}=\pi F \pi^{-1}$, and that $E$ be a homeomorphism that extends continuously to this new boundary on the remainder of the $\pi(U_{ij})$. This agrees with the definition of $E$ off $\displaystyle\bigcup_{i,j}U_{ij}$, and still permits $E$ to be a homeomorphism from each $U_{ij}$ to its respective $V_j$---that is, we still have that $E$ is an essential mating as defined before; we are just being more specific regarding the homeomorphism used in the final step of its construction.

We will consider the essential type construction performed with this essential mating, $E$, and show that it yields a finite subdivision rule. Recall that we need three things for a finite subdivision rule: a tiling, a subdivided tiling, and a subdivision map.

For the tiling $S_\mathcal{R}$, note that ``filling in" the open spaces of a finite, connected, planar graph with open 2-cell tiles guarantees a 2 dimensional CW complex. The 1-skeleton of our tiling starts with two disjoint Hubbard trees, which on their own would be finite and planar, but disconnected. The construction requires that the essential mating is nontrivial with postcritical identifications between trees on $\mathbb{S}^2/\sim_e$ though, so the 1-skeleton is connected and we obtain the desired CW complex. The final requirements for a tiling forbid monogon and digon tiles, but the construction expressly accounts for this by requiring additional marked points to fix potentially errant tiles.


For the subdivision map, we need to show that $E$ restricted to any open cell of $\mathcal{R}(S_\mathcal{R})$ maps homeomorphically onto some open cell of $S_\mathcal{R}$. Since $\mathcal{R}(S_\mathcal{R})$ is obtained by pulling back the structure of $S_\mathcal{R}$ under $E$, this follows from the fact that the critical and postcritical set of $E$ are marked as vertices in $S_\mathcal{R}$. Marked points of $\mathcal{R}(S_\mathcal{R})$ must map to marked points of $S_\mathcal{R}$, and since $E$ is a branched covering it must map homeomorphically on the remaining open tiles and edges.

This leaves checking that the tiling $\mathcal{R}(S_\mathcal{R})$ is a tiling which is a subdivision of $S_\mathcal{R}$. Again, as  $\mathcal{R}(S_\mathcal{R})$ is obtained by pulling back the structure of $S_\mathcal{R}$ under $E$, it will yield a tiling---but it is not obvious that this tiling results from a subdivision of $S_\mathcal{R}$. We will need to check that the open tiles and edges of $\mathcal{R}(S_\mathcal{R})$ resemble open tiles and edges of $\mathcal{R}(S_\mathcal{R})$ which have been subdivided by open edges and vertices. We will obtain this condition if the 1-skeleton of $\mathcal{R}(S_\mathcal{R})$ contains a subdivision of the 1-skeleton of $S_\mathcal{R}$. This will be true if the 1-skeleton of $S_\mathcal{R}$ is forward invariant under $E$.

By the essential construction, note that the 1-skeleton of $S_\mathcal{R}$ is given by points in $\pi(T_\alpha\bigsqcup T_\beta)$. The definition of our essential mating $E$, however, yields that $E|_{\pi(T_\alpha\bigsqcup T_\beta)}=\pi \circ F\circ  \pi^{-1}$. Thus, $E$ maps our 1-skeleton to $\pi\circ F (T_\alpha\bigsqcup T_\beta)$. Recall that the formal mating $F$ acts as $f_\alpha$ on $T_\alpha$ and as $f_\beta$ on $T_\beta$, though. Since Hubbard trees are forward invariant under their associated polynomials, $F$ preserves $T_\alpha\bigsqcup T_\beta$, and so our 1-skeleton is mapped to itself under $E$.

Since we have shown that $E$ acts as a subdivision map from the subdivided tiling $\mathcal{R}(S_\mathcal{R})$ to the tiling $S_\mathcal{R}$, the essential type construction yields a finite subdivision rule.
 \end{proof}

In simpler words, Theorem \ref{thm:construc1} tells us that we will have a problem building a finite subdivision rule using the essential type construction.only whenever two points are identified by $\sim_e$, but their preimages are not. 

\subsection{An example} 
To highlight a case where the essential construction yields a finite subdivision rule, we consider the mating $f_{1/6}\upmodels_e f_{1/6}$. The essential construction prescribes that we start with the disjoint union of Hubbard trees of the two constituent polynomials in the mating, $T_{1/6}$ and $T_{1/6}$, and then take a quotient under the relation $\sim_e$ associated with this mating. The Hubbard tree is presented on the left of Figure \ref{fig:f16f16con1i}, and $T_{1/6}\bigsqcup T_{1/6}/\sim_e$ is shown on the right. (Recall that a pair of external rays adjacent to the same spot on the equator of $\mathbb{S}^2$ will land at $\theta$ and $1-\theta$ on opposing Julia sets in the formal mating. Thus, if there is a $\theta$ and $1-\theta$ pairing of postcritical points on opposing trees, these points collapse under $\sim_e$.) The resulting 1-skeleton yields a two-tile subdivision complex $S_\mathcal{R}$.

\begin{figure}[htb]
\center{\includegraphics{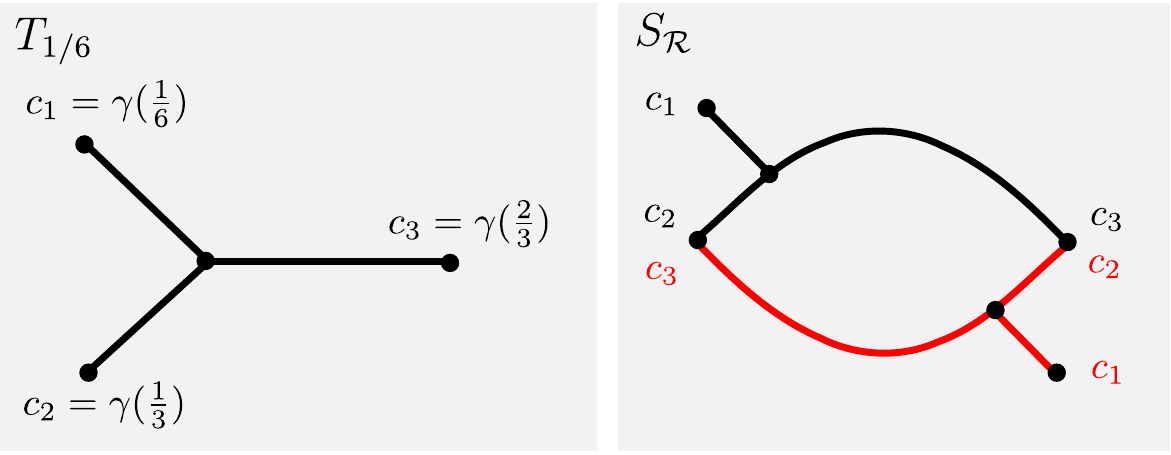}}
\caption{\label{fig:f16f16con1i}{The Hubbard tree for $f_{1/6}$, and the 1-skeleton of the essential type subdivision complex, $S_\mathcal{R}$, for $f_{1/6}\upmodels_ef_{1/6}$}}
\end{figure} 

We now need to take the pullback of $S_\mathcal{R}$ under $E$ to obtain the subdivided complex $\mathcal{R}(S_\mathcal{R})$. It may not be immediately obvious how to determine what the resulting 1-skeleton looks like, but the Hubbard tree structure is helpful here: the preimage of a Hubbard tree under its associated polynomial yields two miniature copies of the tree which map homeomorphically onto the original tree, joined at the critical point. This suggests ``missing limbs" that when filled in will subdivide the tiles of $S_\mathcal{R}$. Noting where each of the marked points maps forward shows where to embed these limbs, since the 1-skeleton of $\mathcal{R}(S_\mathcal{R})$ should map homeomorphically onto the 1-skeleton of $S_\mathcal{R}$ off of the critical point. This yields $\mathcal{R}(S_\mathcal{R})$, as shown in the right side of Figure \ref{fig:f16f16con1ii}. 

\begin{figure}[htb]
\center{\includegraphics{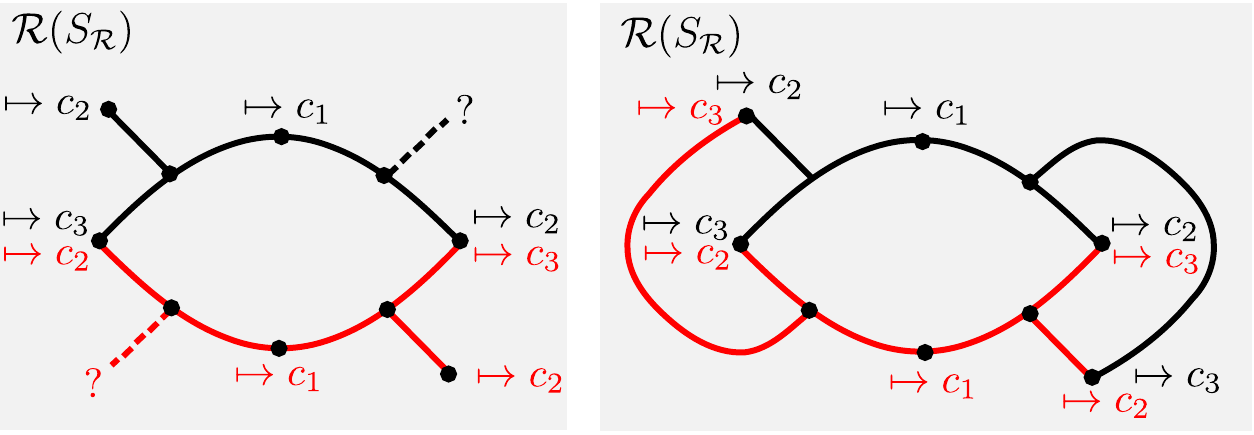}}
\caption{\label{fig:f16f16con1ii}{Determining the essential type subdivided complex, $\mathcal{R}(S_\mathcal{R})$}}
\end{figure} 

An important thing to note in the above example is that we can obtain up to the first  subdivision utilizing the given essential mating map, but that subsequent pullbacks by $E$ do not subdivide in the manner suggested by the original tiles. After the first subdivision we exhaust all of the equivalence classes that collapse to form the quotient space for the essential mating, meaning that the essential mating is not actually a subdivision map for these later iterations. This is precisely the problem that we want to avoid in developing a setting for the essential type construction to admit a finite subdivision rule.

Recall that finite subdivision rules do not require embedded structures or maps to yield a rule, though---combinatorially defined rules are acceptable. In this case, we can use the combinatorial rule implied by the essential construction after the first iteration. Figure \ref{multiplesubs} shows this for the $f_{1/6}\upmodels_ef_{1/6}$ example mentioned above; note how the essential construction yields a two-tile subdivision rule with a quadrilateral and an octagon. When subdividing, the quadrilateral is replaced with an octagon, and the octagon is subdivided into two quadrilaterals and a smaller octagon. This pattern continues for future subdivisions.

\begin{figure}[htb]
\center{\includegraphics{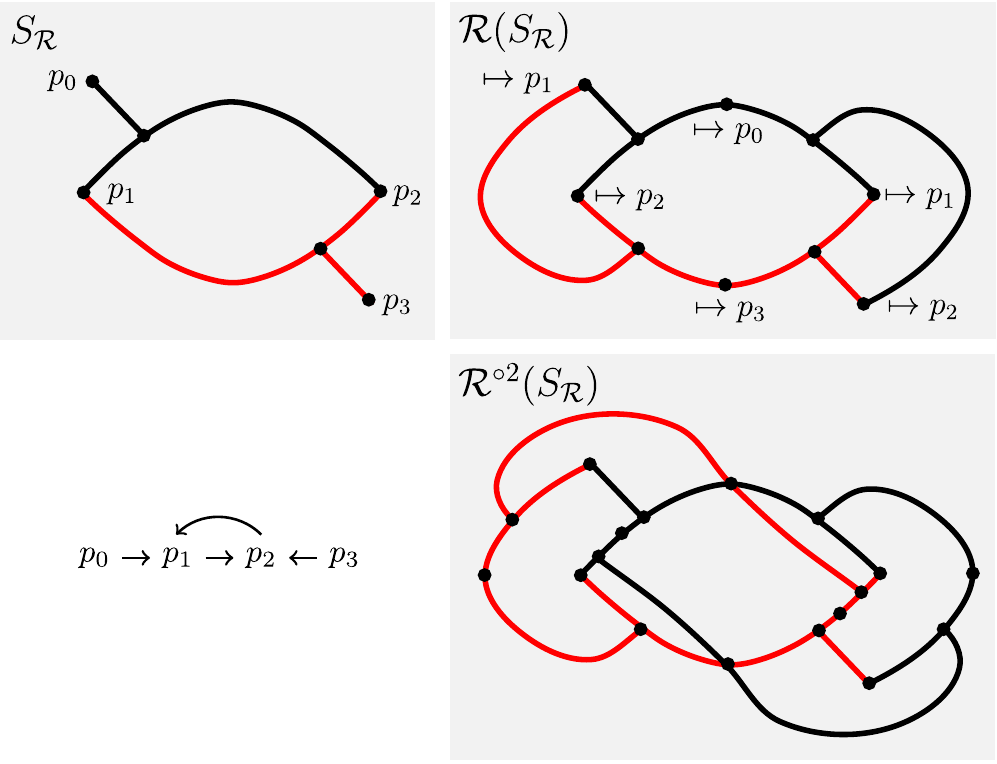}}
\caption{\label{multiplesubs}{Subsequent subdivisions of of $S_\mathcal{R}$ for the mating $f_{1/6}\upmodels_ef_{1/6}$.}}
\end{figure} 

While this subdivision rule will not reflect the behavior of the essential mating after the first subdivision (the subsequent subdivisions would suggest an infinite number of nontrivial equivalence classes of $\sim_e$ as we keep subdividing, which is impossible), it does show us identifications made in the topological mating. Any time the opposing Hubbard tree structures meet reflects some equivalence class of $\sim_t$ collapsing to a point.


\subsection{A non-example} To highlight a less trivial situation in which the essential construction does \emph{not} yield a finite subdivision rule, we will consider the example $f_{7/8}\upmodels_ef_{1/4}$. In Figure \ref{problemex1}, we see the two Hubbard trees needed for the construction with postcritical points and branched points marked, along with the subdivision complex $S_\mathcal{R}$ associated with the essential construction for this mating. For ease of notation in the figures, we set $\gamma(\theta):=\gamma_{7/8}(\theta)$, and $\gamma(\theta)^*:=\gamma_{1/4}(1-\theta)$. When building $S_\mathcal{R}$, it will help to recall that this implies $\gamma(\theta)\sim_e\gamma(\theta)^*$.

\begin{figure}[htb]
\center{\includegraphics{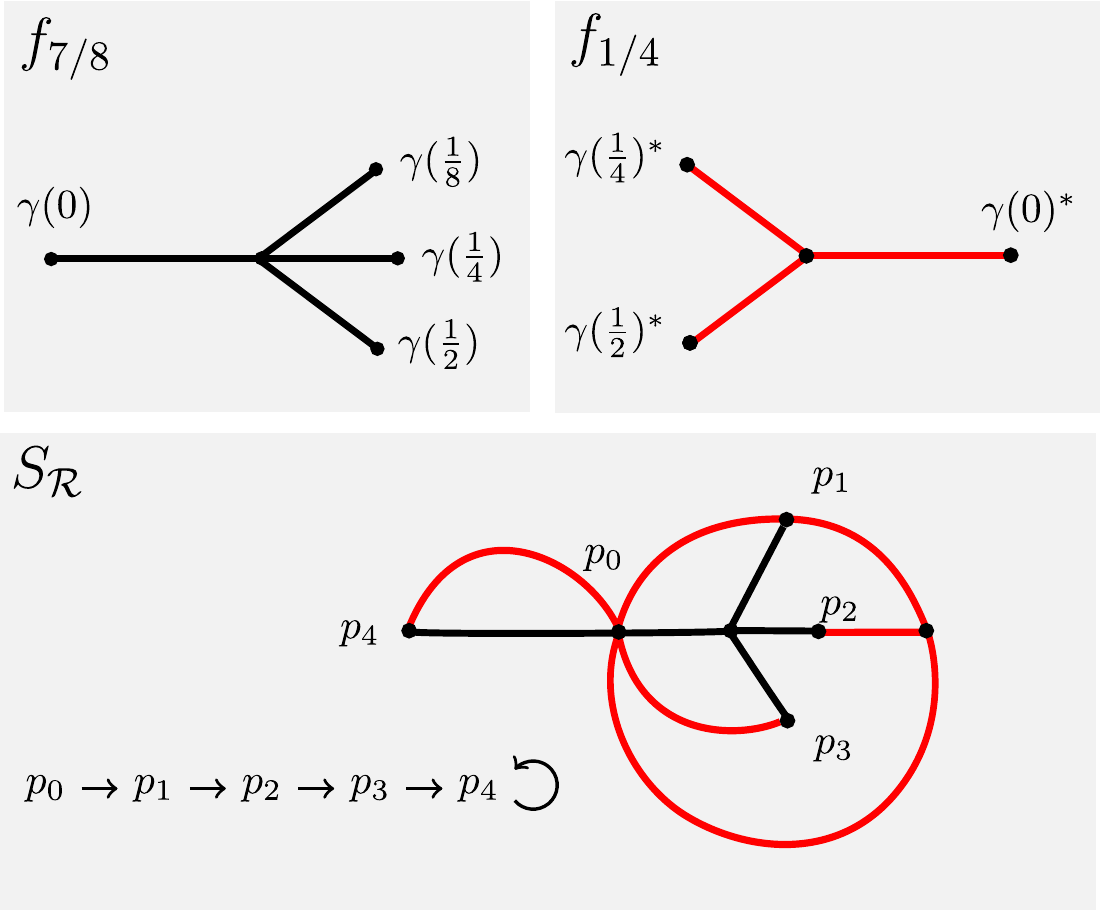}}
\caption{\label{problemex1}{Hubbard trees for $f_{7/8}$ and $f_{1/4}$, along with $S_\mathcal{R}$ as suggested by the essential construction for $f_{7/8}\upmodels_ef_{1/4}$.}}
\end{figure}

The critical portrait for this essential mating suggests a subdivision similar to that given in Figure \ref{problemex2}: first, we note where each of the marked points will map; and second, since we expect that the rule reflects a degree two map we should subdivide 1- and 2-cells as needed to yield a homeomorphic mapping onto $S_\mathcal{R}$. This forces the addition of 4 new edges and 4 new vertices to our structure---but regardless of their placement, no subdivision will have $f_{7/8}\upmodels_ef_{1/4}$ serve as the subdivision map for a subdivision rule. The grey regions highlighted in Figure \ref{problemex2} contain points on the initial Hubbard trees which identify under $\sim_t$ but not $\sim_e$, and whose forward images identify under $\sim_e$. There are two ways to view why this is problematic: first, subdivisions of the initial tiling will not map locally homeomorphically onto $S_\mathcal{R}$ off of the critical points, thus any finite subdivision rule with subdivision complex $S_\mathcal{R}$ cannot have the essential mating as a subdivision map. Alternatively, pullbacks of $S_\mathcal{R}$ under the essential mating are not proper subdivisions. Instead, they possess 1-skeletons that appear to be ``pinched" versions of subdivided 1-skeletons.
 
\begin{figure}[htb]
\center{\includegraphics{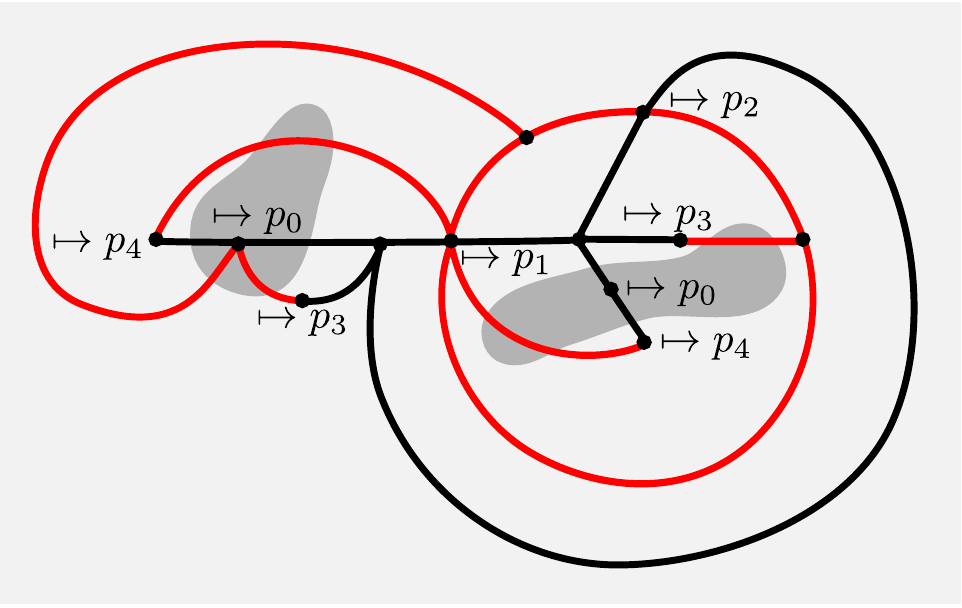}}
\caption{\label{problemex2}{A subdivision of $S_\mathcal{R}$ from Figure \ref{problemex1} that does not map homeomorphically onto $S_\mathcal{R}$.}}
\end{figure} 

Experimentally, the essential construction appears most likely to falter with polynomial pairings like $f_{7/8}$ and $f_{1/4}$ where some equivalence class of $\sim_e$ contains two points from the same Hubbard tree. This is not to say that these kinds of matings cannot be expressed by finite subdivision rules, however. In many cases, minor adaptations can be made to the essential construction in order to produce a rule. One such adaptation is presented in Figure \ref{alternatecon}: since the full critical orbit of $f_{7/8}\upmodels_ef_{1/4}$ is contained in $T_{7/8}/\sim_e$, we can use this as the 1-skeleton for a subdivision complex rather than $T_{7/8}\bigsqcup T_{1/4}/\sim_e$. The proof of Theorem \ref{thm:construc1} implies that if a 1-skeleton is finite, connected, planar, forward invariant, and contains the postcritical set as vertices, then filling in the 1-skeleton with tiles will yield a finite subdivision rule. The subdivision complex in this modified finite subdivision rule is then a $10$-gon which is subdivided into two $10$-gons when pulled back by the essential mating $f_{7/8}\upmodels_ef_{1/4}$.

\begin{figure}[htb]
\center{\includegraphics{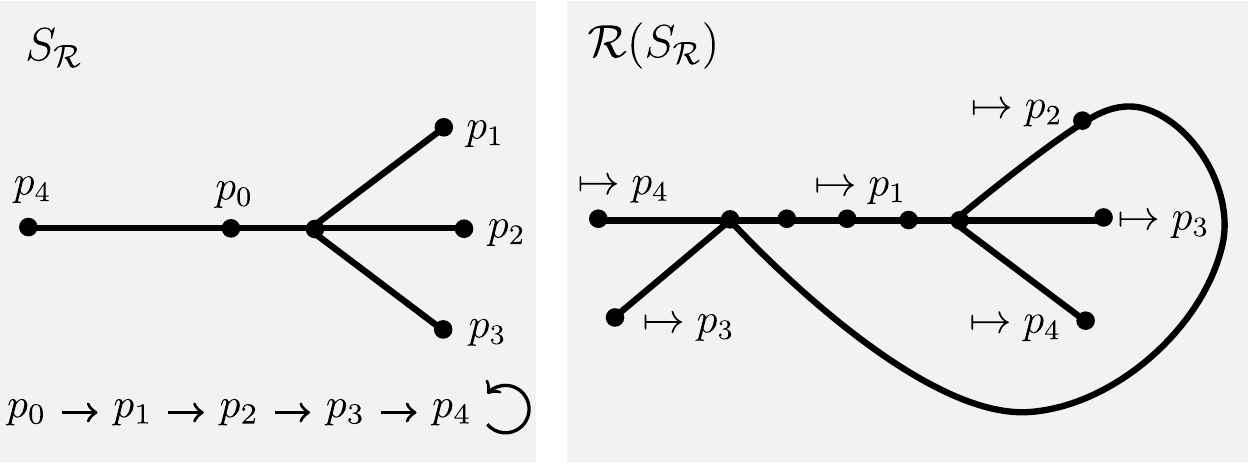}}
\caption{\label{alternatecon}{A finite subdivision rule with subdivision map given by the essential mating $f_{7/8}\upmodels_ef_{1/4}$.}}
\end{figure}


\section{The essential construction and the pseudo-equator}
\label{connections}

In a sense, the essential construction shows us where the ``most important" identifications in a mating are formed first, since we start with the essential mating and then are shown where subsequent preimage identifications must be made on polynomial Julia sets.

This section elaborates on how this technique can provide insights into other means for visualizing and understanding matings.


\subsection{Meyer's pseudocircles} 

In \cite{MATINGEXAMPLES}, Meyer shows that certain postcritically finite rational maps can be viewed as matings and then decomposed into their two constituent polynomials.  If the Julia set of the rational map is a two-sphere, a sufficient condition for such a decomposition is the existence of a \emph{pseudo-equator}:

\begin{definition}\label{pseudoequator}

A homotopy $H: X\times [0,1]\rightarrow X$ is a \emph{pseudo-isotopy} if $H: X\times [0,1)\rightarrow X$ is an isotopy. We will assume $H_0=H(x,0)=x$ for all $x\in X$.

Let $f$ be a postcritically finite rational map, $\mathcal{C}\subseteq \hat{\mathbb{C}}$ be a Jordan curve with $P_f\subseteq\mathcal{C}$, and $\mathcal{C}^1=f^{-1}(\mathcal{C})$. Then we say that $f$ has a \emph{pseudo-equator} if it has a pseudo-isotopy $H: \mathbb{S}^2\times[0,1]$ rel. $P_f$ with the following properties:

\begin{enumerate}
\item $H_1(\mathcal{C})=\mathcal{C}^1$.
\item The set of points $w\in \mathcal{C}$ such that $H_1(w)\in f^{-1}(P_f)$ is finite. (We will let $W$ denote the set of all such $w$.)
\item $H_1:\mathcal{C}\backslash W\rightarrow \mathcal{C}^1\backslash f^{-1} (P_f)$ is a homeomorphism.
\item $H$ deforms $\mathcal{C}$ orientation-preserving to $\mathcal{C}^1$.
\end{enumerate}

\end{definition}

The motivation for the pseudo-equator definition appears forced when approached from the starting point of a rational map, but is quite natural when starting with the mating:

\begin{theorem}\label{secondthm}Let $\mathbb{S}'^2$ denote the quotient space associated with the mating $E=f_\alpha\upmodels_ef_\beta$, and let $P_E$ denote the postcritical set of $E$. If there exists some Jordan curve $\mathcal{C}$  on $X$ which contains $P_E$ and separates $(T_\alpha/\sim_e) \backslash P_E$ from $(T_\beta/\sim_e) \backslash P_E$, then $E$ has a pseudo-equator.\end{theorem}

\begin{proof}
Consider the pullback of $\mathcal{C}$ under $E$, $\mathcal{C}^1$. Since $\mathcal{C}$ contains the critical values of $E$, $\mathcal{C}^1$ must pass through the two critical points of $E$. Locally, the pullback resembles an X at the critical points because $E$ is a degree 2 map---and these are the only locations that the pullback has this shape, since there are only two critical points. 

Since $E$ is a branched covering map, there are a limited number of options for the topological shape of the pullback $\mathcal{C}^1$ since $\mathcal{C}^1$ may only cross itself twice. The options resemble those given in Figure \ref{fig:curvecases}, up to inclusion of additional components that are Jordan curves.

\begin{figure}[htb]
\center{\includegraphics{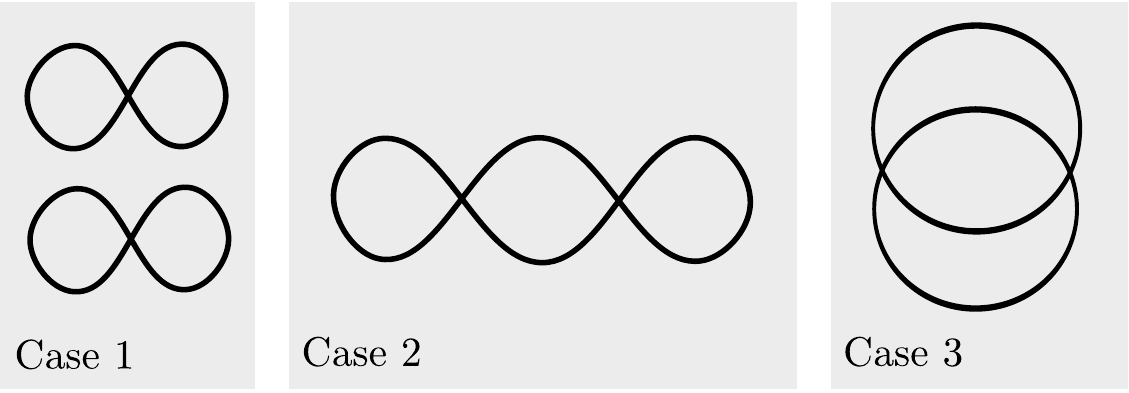}}
\caption{\label{fig:curvecases}{Possible pullbacks of $\mathcal{C}$ under a branched covering map.}}
\end{figure}

These are possibilities for a generic branched covering not specific to $E$, however. The first case in Figure \ref{fig:curvecases} cannot be the pullback because $\mathbb{S}'^2\backslash \mathcal{C}^1$ contains too many components: $E$ is a degree two map, and acts homeomorphically off the critical set. This means that we should expect $\mathbb{S}'^2\backslash\mathcal{C}^1$ to have 4 components. This line of reasoning also rules out the possibility of adjoining additional Jordan components to any of the cases in Figure \ref{fig:curvecases}.

The second case we can rule out using a similar line of reasoning: we can examine where segments of the pullback will map based on where the endpoints map. The segments on either end start and end at a critical point, which means the image of these segments under $E$ must start and end at a critical value. These end segments, when paired with their respective critical points, must map onto $\mathcal{C}$. The two segments in the middle when paired with the critical points must also map onto $\mathcal{C}$. This suggests that $E$ is at minimum a degree 3 map, which is not the case.

We are left with the pullback resembling the the last case of Figure \ref{fig:curvecases}. Since $E$ acts homeomorphically off of the critical set, we expect a mapping behavior much like that expressed in Figure \ref{pullbackcurve}. In this figure, blue lines denote the indicated curve and dots mark critical points. The bolded black and red lines mark Hubbard trees, with dashing to denote that we are only showing local behavior of the tree near the critical point. Notice that if we ``sliced" $\mathcal{C}^1$ along the Hubbard trees, we'd obtain a curve that could be deformed in an orientation preserving manner to $\mathcal{C}$. This deformation hints at the desired pseudo-isotopy $H$.

\begin{figure}[htb]
\center{\includegraphics{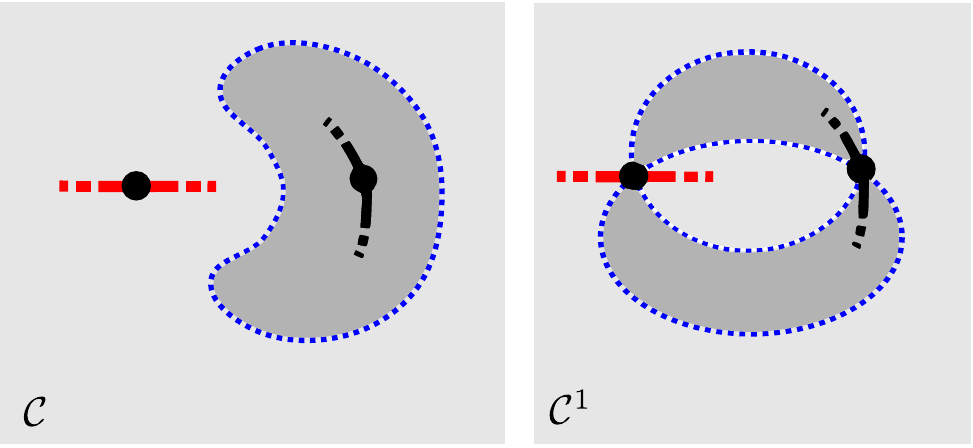}}
\caption{\label{pullbackcurve}{$\mathcal{C}$ and its pullback, shown with local behavior of Hubbard trees near the critical points of $E$.}}
\end{figure}

For brevity we leave the explicit construction of $H$ to the interested reader, but offer the following comments: $H$ should be constructed to avoid mapping arcs of $\mathcal{C}$ to single points of $\mathcal{C}^1$. Further, borrowing notation from the definition of pseudo equator above, we expect that $W=P_E$. This should guarantee conditions (2) and (3) in Definition \ref{pseudoequator}.

\end{proof}

\subsection{An example, continued}
Theorem \ref{secondthm} implies the following method for finding pseudo-equators associated with a mating: if $\Gamma$ is homotopic to the equator on $\mathbb{S}^2$ relative to $T_\alpha$ and $T_\beta$, then $\mathcal{C}=\Gamma/\sim_e$ generates a pseudo-equator when $\mathcal{C}$ is a Jordan curve. It is thus reasonably straightforward to visualize the pseudo-equator on particular matings by using the essential construction: form a finite subdivision rule using the essential construction, and on $S_\mathcal{R}$ draw a curve $C$ through the postcritical points such that $\mathbb{S}'^2 \backslash C$ contains two components---the closure of each containing the Hubbard tree of a polynomial in the mating. If $C$ is a Jordan curve, $C$ generates a pseudo-equator, and the subdivision map shows us how Meyer's two-tiling subdivides, as in Figure \ref{meyerex}.

\begin{figure}[htb]
\center{\includegraphics{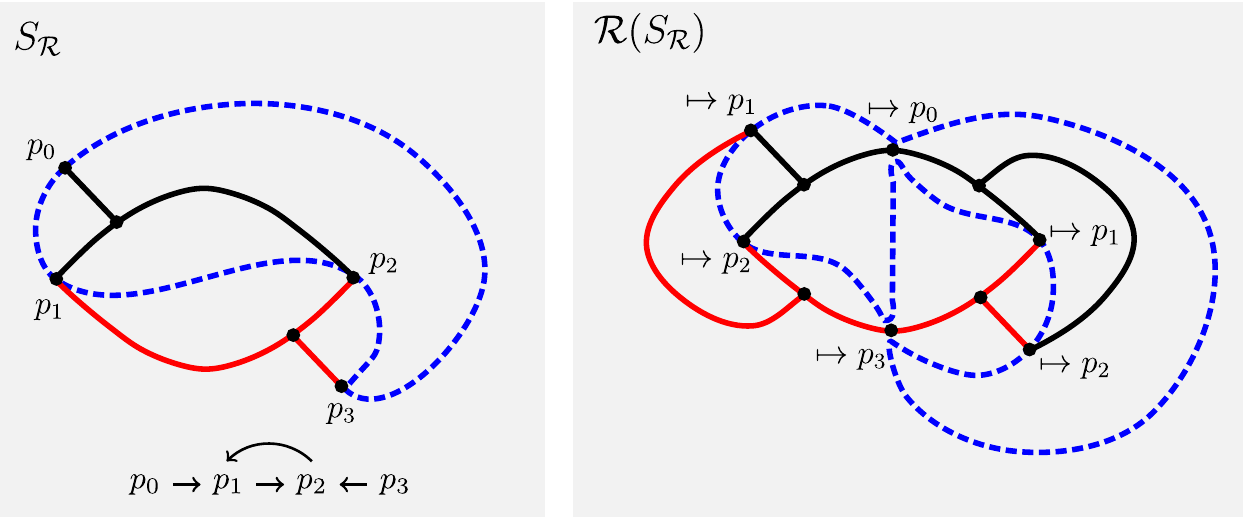}}
\caption{\label{meyerex}{The pseudo-equator associated with $f_{1/6}\upmodels_ef_{1/6}$. $C$ is marked in blue on the left. The pullback of $C$ under this mating is marked in blue on the right.}}
\end{figure}

With consideration for edge replacements in the pullback, the pseudo-equator provides a means for recovering information on the polynomial pair associated with the mating. Although it should be clear in this $f_{1/6}\upmodels_ef_{1/6}$ example that the polynomials associated with the pseudo-equator are two copies of $f_{1/6}$, we can confirm the decomposition for C using the methods given in \cite{MATINGEXAMPLES}.

First, label the postcritical vertices along the pseudo-equator as $p_0,...,p_n$. We then label each edge from $p_i$ to $p_{i+1 (\mod  n+1)}$ as $E_i$, and determine the edge replacement matrix $(a_{ij})$ of the pseudo isotopy where $a_{ij}$ is the number of distinct sub-edges of $H_1(E_i)$ which map to $E_j$ . The edge replacement matrix for the example in Figure \ref{meyerex} is

$$\left[\begin{array}{rrrr} 0&1&0&0\\
1&0&1&1\\
0&1&0&0\\
1&0&1&1\\
\end{array}\right].$$

The degree of the mating corresponds to the leading eigenvalue of the edge replacement matrix, which is 2. When normalized so that the sum of entries is 1, the corresponding eigenvector is $v=\left[\frac{1}{6}\ \ \frac{1}{3}\ \ \frac{1}{6}\ \ \frac{1}{3}\ \right]^T$. The entries $v_0, v_1,...$ of $v$ then correspond to the lengths of edges $E_0, E_1,...$ on the pseudo-equator, which in turn determines spacing of the marked postcritical points $p_i$. 

Since the spacing between these points does not immediately provide information about the mating, we let the function $\theta:\{p_0,...,p_n\}\rightarrow[0,1)$ denote the external angle associated with each postcritical point with respect to one of the polynomial Hubbard trees (say, the black one in Figure \ref{meyerex}). This function must satisfy two properties: first by tracking lengths of edges, that 

$$\theta(p_i)=\theta(p_0)+\displaystyle\sum_{k=1}^iv_k,$$ and second, we require that 

$$\theta(p_i)=\theta\circ E(p_i)-\theta(p_i) \pmod 1$$ due to the Carath\'{e}odory semiconjugacy associated with the mating. Simple computation allows us to obtain that for the current example, $\theta(p_0)=\frac{1}{6}, \theta(p_1)=\frac{1}{3},\theta(p_2)=\frac{2}{3},$ and $\theta(p_3)=\frac{5}{6}.$ The Carath\'{e}odory semiconjugacy suggests that $p_0$ and $p_3$ are our critical values. Since the external angle is given with respect to the black polynomial, this means that only one of $\theta(p_0)$ or $\theta(p_3)$ may be taken as the correctly oriented angle associated with this polynomial, and that the other is given in reverse orientation. If we choose $p_0$ to have a correctly oriented angle $\frac{1}{6}$, this means that $p_3$ has external angle when oriented to the red polynomial of $1-\theta(p_3)=\frac{1}{6}$. Thus, we confirm that the pseudo-equator is given by $f_{1/6}$ mated with itself. 

\subsection{When pseudo-equators do not exist}
Not all non-hyperbolic matings have pseudo-equators. A potential reason is that the path $C$ is not always a Jordan curve--any time $\sim_e$ contains equivalence classes that include multiple postcritical or critical points from one of the polynomials in the mating, the equator $\Gamma$ is pinched to form $C$. This falls outside of the scope of the definition for a pseudo equator, which concerns the deformation of a Jordan curve.  For instance, the example given in \cite{MATINGEXAMPLES} for $f_{1/6}\upmodels f_{13/14}$ presents with subdivision complex $S_\mathcal{R}$ and $C$ as shown in Figure \ref{meyerex3}. Notice the pinching of the blue equator curve due to the postcritical identifications on $f_{13/14}$.

\begin{figure}[htb]
\center{\includegraphics{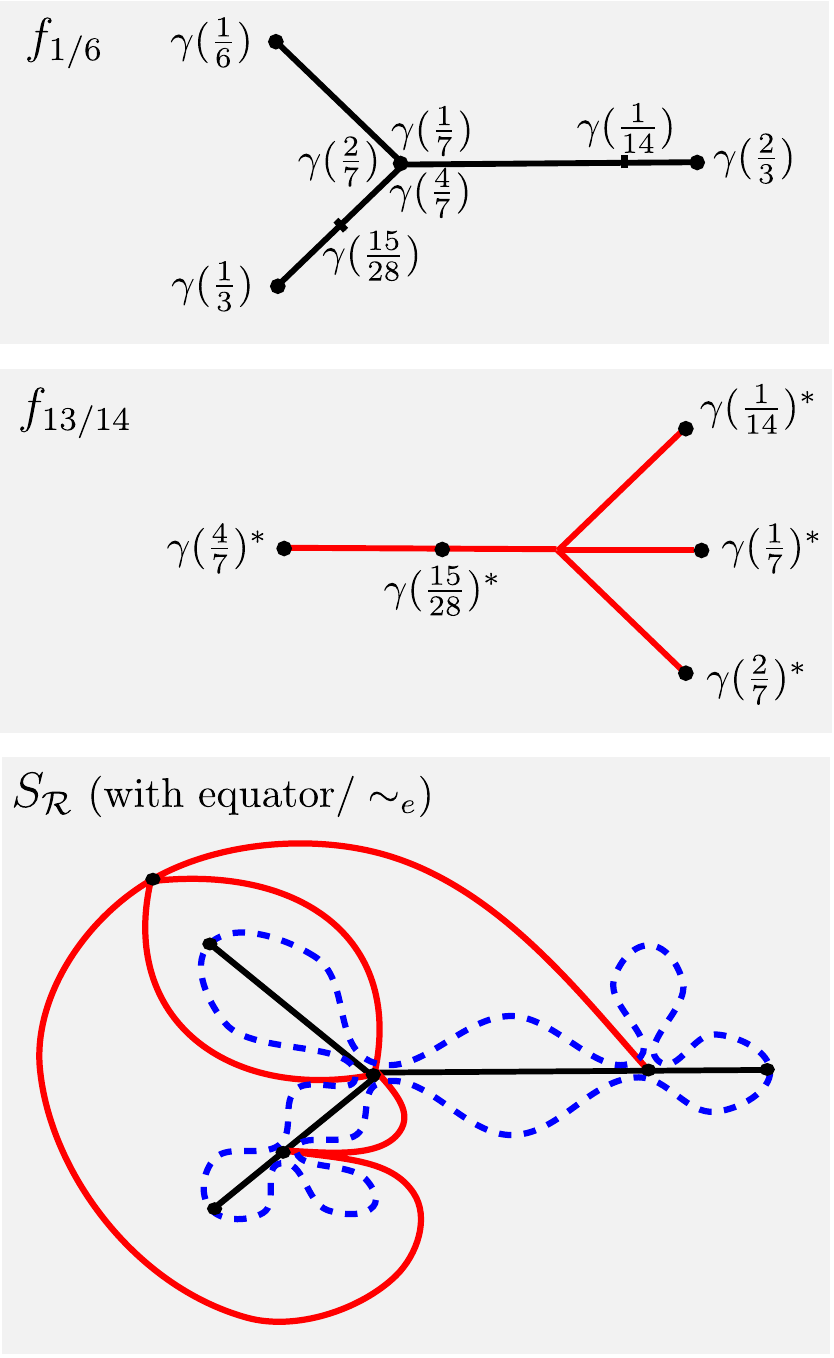}}
\caption{\label{meyerex3}{The ``pseudo-equator" is pinched by $\sim_e$ into a non-Jordan curve.}}
\end{figure}

\subsection{Implications and Future Work}
The essential finite subdivision rule constructions provide an alternative model for matings of critically preperiodic quadratic polynomials. Further, they are a useful tool for visualizing basic dynamics and modeling the mapping properties of certain matings---When paired with Bruin and Schleicher's algorithms from \cite{HUBBARDTREES}, the essential construction is simple enough that many elementary function pairings with few postcritical points can have their mapping behaviors sketched without the aid of a computer.  

In addition, these constructions serve as complementary to work in the current literature: In \cite{MEDUSA}, the Medusa algorithm is provided for obtaining rational maps from matings of quadratic polynomials, but the algorithm eventually diverges in the case of non-hyperbolic pairings.  It is the author's belief that the finite subdivision rule constructions in this paper could be used to modify the Medusa algorithm in a way that would yield rational maps from matings of non-hyperbolic polynomials.

In \cite{MATINGEXAMPLES}, the relationship between rational maps and matings is only stressed with the existence of an equator or pseudo-equator, to the exclusion of structures such as those highlighted in Figure \ref{meyerex3}. As highlighted in the above examples, two-tilings generated by the essential construction have potential to show how non-hyperbolic mated maps are related to different space-filling curves on the two-sphere: the quotient of the equator on $\mathbb{S}^2$ with respect to $\sim_t$ is a topological two-sphere, and 1-skeletons of subdivisions of the two-tiling give subsequent approximations to this quotient space. The essential construction and these two-tilings should provide further insight on the conditions in which postcritically finite rational maps can be realizable as matings, and suggest alternative structures to consider when rational maps do not have pseudo-equators.

\section*{Acknowledgements}

Many of the Julia set graphics throughout the paper were created with the assistance of the dynamics software Mandel 5.8. (See \cite{MANDEL}.) The author would also like to extend sincere thanks to William Floyd for his guidance, as much of the supporting work presented in this paper was done under his advising at Virginia Tech.

\end{document}